\documentclass[a4paper,english]{amsart}
\usepackage[latin9]{inputenc}
\usepackage{babel}
\usepackage{verbatim}
\usepackage{amstext}
\usepackage{amsthm}
\usepackage{amssymb}
\usepackage[unicode=true,
 bookmarks=false,
 breaklinks=false,pdfborder={0 0 1},backref=section,colorlinks=false]
 {hyperref}

\makeatletter

\pdfpageheight\paperheight
\pdfpagewidth\paperwidth

\numberwithin{equation}{section}
\numberwithin{figure}{section}
\theoremstyle{plain}
\newtheorem{thm}{\protect\theoremname}
\theoremstyle{plain}
\newtheorem{lem}[thm]{\protect\lemmaname}
\theoremstyle{definition}
\newtheorem{defn}[thm]{\protect\definitionname}
\theoremstyle{plain}
\newtheorem{prop}[thm]{\protect\propositionname}
\theoremstyle{definition}
\newtheorem{example}[thm]{\protect\examplename}
\theoremstyle{plain}
\newtheorem{conjecture}[thm]{\protect\conjecturename}
\theoremstyle{plain}
\newtheorem*{conjecture*}{\protect\conjecturename}
\theoremstyle{remark}
\newtheorem{claim}[thm]{\protect\claimname}
\theoremstyle{remark}
\newtheorem{rem}[thm]{\protect\remarkname}
\theoremstyle{plain}
\newtheorem{cor}[thm]{\protect\corollaryname}

\usepackage[foot]{amsaddr}

\usepackage{amstext}

\DeclareTextSymbolDefault{\textquotedbl}{T1}

\usepackage{comment}

\theoremstyle{plain}

\DeclareMathOperator{\OF}{\Omega_F}

\DeclareMathOperator{\OS}{\Omega_S}
\DeclareMathOperator{\dOS}{\partial\Omega_S}
\DeclareMathOperator{\dO}{\partial\Omega}
\DeclareMathOperator{\Lp}{L}

\DeclareMathOperator{\H1}{H}
\DeclareMathOperator{\C}{C}
\DeclareMathOperator{\Div}{div}

\newcommand*{\dx}{\mathop{}\!\mathrm{d}}

\providecommand{\corollaryname}{Corollary}
\providecommand{\definitionname}{Definition}
\providecommand{\lemmaname}{Lemma}
\providecommand{\remarkname}{Remark}
\providecommand{\theoremname}{Theorem}

\makeatother

\providecommand{\claimname}{Claim}
\providecommand{\conjecturename}{Conjecture}
\providecommand{\corollaryname}{Corollary}
\providecommand{\definitionname}{Definition}
\providecommand{\examplename}{Example}
\providecommand{\lemmaname}{Lemma}
\providecommand{\propositionname}{Proposition}
\providecommand{\remarkname}{Remark}
\providecommand{\theoremname}{Theorem}

\begin{document}
\title[Stabilty of the fluid-elastic semigroup]{ strong stability and the Schiffer Conjecture for the fluid-elastic
semigroup }
\author{Karoline Disser}
\address{Universität Kassel\\
 Institut für Mathematik \\
 Heinrich-Plett-Stra{ß}e 40 \\
 34132 Kassel, Germany }
\email{kdisser@mathematik.uni-kassel.de}
\begin{abstract}
In a series of papers, Avalos and Triggiani established the fluid-elastic
semigroup for the coupled Stokes-Lamé system modelling the coupled
dynamics of a linearly elastic structure immersed in a viscous Newtonian
fluid. They analyzed the spectrum of its generator and proved that
the semigroup is strongly stable, if the domain of the structure satisfies
a geometric condition, i.e. it is not a \emph{bad domain}. We extend
these results in two directions: first, for bad domains, we prove
a decomposition of the dynamics into a strongly stable part and a
\emph{pressure wave}, a special solution of the Dirichlet-Lamé system,
that can be determined from the initial values. This fully characterizes
the long-time behaviour of the semigroup. Secondly, we show that the
characterization of bad domains is equivalent to the Schiffer problem.
This strengthens the conjecture that balls are the only bad domains
and establishes a direct connection to geometric analysis. We also
discuss implications for associated nonlinear systems. 
\end{abstract}

\thanks{The author gratefully acknowledges financial support by DFG project
FOR 5528.}

\maketitle
\global\long\def\Ls{L_{S}}%

\global\long\def\XsK{\mathbf{X}}%

\global\long\def\AsK{\mathcal{A}}%

\global\long\def\ScsK{\mathcal{S}}%

\global\long\def\Xs{\bar{\mathbf{X}}}%

\global\long\def\As{\bar{\mathcal{A}}}%

\global\long\def\Scs{\bar{\mathcal{S}}}%

\global\long\def\Es{\mathbf{E}}%

\global\long\def\Ws{\mathbf{W}}%

\section{Introduction}

We consider the linear system

\begin{equation}
\begin{cases}
\begin{array}{rcll}
\dot{u}-\Div(S(u,p)) & = & 0 & \textrm{in }(0,T)\times\OF,\\
\Div(u) & = & 0 & \textrm{in }(0,T)\times\OF,\\
S(u,p)n & = & L(\xi)n & \textrm{on }(0,T)\times\dOS,\\
u & = & \dot{\xi} & \textrm{on }(0,T)\times\dOS,\\
u & = & 0 & \text{on }(0,T)\times\dO,\\
\ddot{\xi}-\textrm{div}(L(\xi))+\xi & = & 0 & \textrm{in }(0,T)\times\OS,\\
u(0) & = & u_{0} & \textrm{in }\OF,\\
\xi(0) & = & \xi_{0} & \textrm{in }\OS,\\
\dot{\xi}(0) & = & \xi_{1} & \textrm{in }\OS,
\end{array}\end{cases}\label{eq:linsys}
\end{equation}
modeling the dynamic interaction of a linear elastic structure with
an incompressible Newtonian fluid in a given time interval $(0,T)$,
$T>0$. Here, $\Omega\subset\mathbb{R}^{3}$ is a bounded Lipschitz
domain, and the elastic structure is located at the bounded Lipschitz
domain $\OS$, completely immersed in the fluid in the Lipschitz domain
$\OF$. This condition is expressed by the relation $\overline{\OS}\cup\OF=\Omega$
and we assume that there is no contact of structure and outer boundary
at this level of the formulation, $\overline{\OS}\cap\partial\Omega=\emptyset$.
The interface $\dOS$ is where elastic structure and fluid meet. Here,
the transmission boundary conditions of continuity of velocities and
continuity of forces hold. The exterior unit normal vector field of
$\OF$ at $\dOS$ and $\partial\Omega$ is denoted by $n$. The unkowns
are the fluid velocity $u\colon(0,T)\times\OF\to\mathbb{R}^{3}$,
fluid pressure $p\colon(0,T)\times\OF\to\mathbb{R}^{3}$ and elastic
displacement $\xi\colon(0,T)\times\OS\to\mathbb{R}^{3}$. We denote
the symmetric gradient of a vector field $v$ by 
\[
D(v):=\frac{1}{2}\left(\nabla v+(\nabla v)^{T}\right).
\]
The Newtonian fluid stress tensor is
\[
S(u,p):=2\nu D(u)-p\mathrm{Id}
\]
with given viscosity $\nu>0$ and the Lamé elastic stress tensor is
given by
\[
L(\xi):=\lambda_{0}D(\xi)+\lambda_{1}\Div(\xi)\mathrm{Id},
\]
with Lamé constants $\lambda_{0},\lambda_{1}>0$. Here, as usual,
$\nabla$ and $\Div$ are spatial gradient and (line-wise) divergence
and we use the notation $\dot{v},\ddot{v}$ to denote first and second
partial derivatives with respect to the time variable.

System \eqref{eq:linsys} arises as the linearization of the non-linear
coupled system of fluid-elastic interaction that includes the dynamic
change of the fluid domain according to the displacement of the structure.
This system is discussed in some more detail in Subsection \ref{subsec:fullynonlinear}. 

In a series of papers, \cite{4AT2009,5AT,8AT2009,Astrongstab,AT2,ATboundary,ATuniformstab},
Avalos and Triggiani studied system \eqref{eq:linsys} and the corresponding
simpler system with the Lamé part replaced by the wave equation. In
particular, by applying the Lumer-Phillips theorem, they established
the existence of a strongly continuous semigroup of contractions associated
to this problem and characterized its long-time behaviour based on
properties of the spectrum of its generator and the Arendt-Batty-Lyubich-Vu
criterion \cite{ArendtBatty1988,LyubichPhong1988}. For a large class
of \emph{good domains} $\OS$, as well as in the case of additional
damping, they showed strong stability of the semigroup \cite{5AT,ATuniformstab,ATboundary,8AT2009,4AT2009,AT2,Astrongstab}.
Our characterization of long-time asymptotic behaviour of solutions
extends these stability results to a related statement in the general
case of \emph{bad domains}, when non-trivial time-periodic solutions
that we call \emph{pressure waves} may occur. It fully determines
the attractor of the semigroup and provides convergence to a limit
solution that can be read off from the initial data. Pressure waves
were also identified in the works of Avalos and Triggiani, but not
studied in the long-time limit. 

\subsection{Organization and discussion of the main ideas}

\subsubsection*{Strong stability on bad domains}

In Section \ref{sec:ATsemigroup}, we recall the results of Avalos
and Triggiani, in a form adapted to the further analysis (not all
results can be covered). In Section \ref{sec:Invariance}, we study
pressure waves in some detail and show that the fluid-elastic semigroup,
after factorization with the one-dimensional kernel of the generator
$\AsK$, respects the eigenvalue decomposition corresponding to pressure
wave states. Thus, on bad domains, the dynamics can be decomposed
and we take out pressure waves. By exploiting the spectral properties
of $\AsK$ shown by Avalos and Triggiani, we can apply the Arendt-Batty-Lyubich-Vu
criterion again on the remaining subspace semigroup to obtain strong
stability also in this case.

\subsubsection*{Bad domains are Schiffer domains and vice versa}

In Section \ref{sec:Schiffer}, we turn to the related but conceptually
different topic of characterizing bad domains. We show that a domain
$\OS$ is bad if and only if it is a Schiffer domain, which means
that the overdetermined \emph{scalar} Neumann eigenvalue problem with
additional constant Dirichlet boundary data, 
\[
\begin{split}\begin{cases}
\begin{array}{rcll}
-\Delta u & = & \mu_{S}u, & \text{in }\OS,\\
\partial_{n}u & = & 0, & \text{on }\dOS,\\
u & = & c, & \text{on }\dOS.
\end{array}\end{cases}\end{split}
\]
has a non-trivial solution. The Schiffer conjecture is a well-known
open problem in geometric analysis. It states that the only Schiffer
domains are the balls. A brief discussion of literature on this problem
is given at the beginning of Section \ref{sec:Schiffer}. Our result
strongly connects the analysis of fluid-structure interaction to geometric
analysis and further topics in harmonic analysis via the Pompeiu problem.
Clearly, this helps to frame the problem in the fluid-elastic context.
It is unclear whether the connection can help to solve the Schiffer
or Pompeiu problems, because the connection is elementary. In the
calculations, we recover a sufficient condition for $\OS$ to be a
ball in terms of a third-order Neumann condition for $u$ in Corollary
\ref{cor:3neumannBC}. It will be interesting to explore the connection
further. 

\subsubsection*{Implications for related nonlinear systems }

With the above results, the linear theory for the fluid-(linearly)
elastic system is fairly complete: global well-posedness is known
in a mild and a strong setting, cf.~\cite{DisserLuckas2025} for
the second result, and the asymptotic behaviour is fully characterized
for both good and bad domains. The geometric characterization of bad
domains can be ,,outsourced'' to geometric analysis. So it makes sense
to look at first implications for the non-linear counterparts in Section
\ref{sec:nonlinear}. Pressure waves also solve the corresponding
nonlinear systems without damping. This means that global existence
and stability results must account for the corresponding time-periodic
dynamics. The linear theory gives some of the tools for this analysis,
like invariance. For the semilinear system that includes fluid convection,
but keeps the domains fixed, this is sufficient to fully characterize
the long-time behaviour as well, Theorem \ref{thm:longtimenonlinear}.
For the fully nonlinear system, further investigations are needed. 

\subsection{Related results on fluid-elastic interaction }

In addition to the works of Avalos and Triggiani, the linear system
is studied in \cite{GMZZ2014}, in the context of the occurence of
resonance in parabolic-hyperbolic coupled systems. Regarding the existence
of solutions for related nonlinear systems, we refer to \cite{Boulakia2007,CS2005,KT2012,IKLT2017,RV2014,KO2023,KO2024,BKS2024}.
Note that these results include some viscosity on the elastic parts,
so pressure waves do not occur. In some cases, exponential convergence
to the rest state is obtained \cite{IKLT2017,KO2023,KO2024}. 

In \cite{DL2022,DisserLuckas2025}, we proved existence of strong
solutions for system \ref{eq:linsys} with fluid convection, locally
in time and globally for small data, and we characterized the long-time
behaviour. With the help of the present analysis, these results can
be slightly improved. This is discussed in Subsection \ref{subsec:resultsluckas}.

There is also a large body of literature on related models that concern
the interaction with elastic shells, elastic beams, compressible or
inviscid fluids. We refer e.g. to \cite{CS2006,LR2014,MMNRT2022,GH2016,GHL2019,MC2015,KT2024,AKT2025,BKS2024compressible,MRR2020}
and references therein. 

\section{The fluid-elastic semigroup of Avalos and Triggiani\label{sec:ATsemigroup}}

\subsection{Definition of the semigroup and its generator}

Avalos and Triggiani \cite{Astrongstab,5AT,4AT2009,8AT2009,AT2} established
the existence of a strongly continuous semigroup of contractions,
$(\ScsK(t))_{t>0},$
\[
\ScsK(t)\colon(\xi_{0},\xi_{1},u_{0})\in\XsK\mapsto(\xi(t),\dot{\xi}(t),u(t))\in\XsK,
\]
with generator 
\[
\AsK\colon D(\AsK)\subset\XsK\to\XsK,
\]
associated to system \eqref{eq:linsys} in the energy space 
\[
\XsK:=H^{1}(\OS)^{3}\times L^{2}(\OS)^{3}\times\Ls,
\]
where $L^{2}(O)$ and $H^{m}(O),m\in\mathbb{N}$, denote the usual
Lebesgue and Sobolev spaces on a domain $O$. The exponent $^{3}$
will often be ommited for simplicity. The space $H^{1}(\OS)^{3}$
is endowed with the inner product 
\begin{equation}
(\xi,\tilde{\xi})_{H^{1}(\Omega)}=\int_{\OS}\xi\cdot\tilde{\xi}+\int_{\OS}\sigma(\xi):\nabla\tilde{\xi}=\int_{\OS}\xi\cdot\tilde{\xi}+\int_{\OS}\sigma(\xi):D(\tilde{\xi}).\label{eq:innerproduct}
\end{equation}
By Korn's inequality, the induced norm is equivalent to the standard
one. The space
\[
\Ls:=\left\{ u\in L^{2}(\OF)^{3}:\Div u=0,u\cdot n|_{\dOS}=0\right\} 
\]
of weakly divergence-free vector fields with partial zero normal component
at the outer boundary is defined based on the elimination of the pressure
variable from \eqref{eq:linsys}. It is done in a natural way by an
harmonic function that satisfies the correct boundary conditions.
More precisely, the following construction is used, \cite[Section 2.2]{8AT2009}.
\begin{lem}
Given $(\xi,\zeta,u)\in D(\AsK)$, there exists a unique harmonic
pressure function $p=P_{M}(u)+P_{D}(\xi)\in L^{2}(\OF)$, where $P_{M}(u)$
and $P_{D}(\xi)$ solve the inhomogeneous mixed Dirichlet-Neumann
problems

\[
\begin{cases}
\begin{array}{rcll}
\Delta P_{M} & = & 0, & \text{in }\OF,\\
(\nabla P_{M})n & = & (\Delta u)|_{\partial\Omega}n, & \text{on }\partial\Omega,\\
P_{M} & = & n^{T}D(u)n, & \text{on }\dOS,
\end{array}\end{cases}
\]
and
\[
\begin{cases}
\begin{array}{rcll}
\Delta P_{D} & = & 0, & \text{in }\OF,\\
(\nabla P_{D})n & = & 0, & \text{on }\partial\Omega,\\
P_{D} & = & -n^{T}L(\xi)n, & \text{on }\dOS,
\end{array}\end{cases}
\]
respectively, in a weak sense. 
\end{lem}

With this construction, the generator $\AsK$ can be given more explicitly,
i.e.

\[
\AsK\left(\begin{array}{c}
\xi\\
\zeta\\
u
\end{array}\right)=\left(\begin{array}{ccc}
0 & -\mathrm{Id} & 0\\
-\Div L+\mathrm{Id} & 0 & 0\\
\nabla P_{D} & 0 & -\Delta+\nabla P_{M}
\end{array}\right)\left(\begin{array}{c}
\xi\\
\zeta\\
u
\end{array}\right).
\]
The characterization of the domain $D(\AsK)$ will be discussed further
below. Even though it is elliminated from the semigroup, the pressure
function plays an important role in fluid-structure interaction. For
example, in Subsection \ref{subsec:Definition-of-Pressure} , it is
shown that spatially constant (hydrostatic) pressure may still be
associated with interesting elastic dynamics. There is some confusion
about this point in the wider literature, because the situation for
the fluid is unlike the one created by the classical Navier-Stokes
equations in the following sense: the Dirichlet condition at $\dOS$
in \eqref{eq:linsys} shows that the pressure is fully determined
by the fluid-elastic dynamics, not just up to a constant. Or, put
the other way around: changing the hydrostatic pressure will change
the elastic dynamics. 

Nevertheless, there is an invariant scalar quantity that needs to
be factorized out of the dynamics: it is associated to the choice
of $\OS$ as a suitable reference configuration. More precisely, consider
the quantity
\[
K_{\xi}(t):=\int_{\dOS}\xi(t)\cdot n.
\]
If $(u,\xi)$ solve system \eqref{eq:linsys} in a sufficiently smooth
way, then
\[
\dot{K}_{\xi}(t)=\int_{\dOS}\dot{\xi}(t)\cdot n=\int_{\dOS}u(t)\cdot n=\int_{\OF}\Div u(t)+\int_{\partial\Omega}u(t)\cdot n=0,
\]
so $K_{\xi_{0}}:=\int_{\dOS}\xi_{0}\cdot n$ remains invariant. This
invariance was observed by Avalos and Triggiani and its rigorous proof
for the semigroup is cited below in Lemma \ref{lem:Kxi}. Due to $K_{\xi}(t)=\int_{\OS}\Div\xi(t)$,
this quantitiy can be seen as a linearized measure of the mean volume
change induced by $\xi$ on $\OS$. As the fluid is incompressible,
no mean volume change is possible. It is thus reasonable to assume
that, at initial time, $|\xi_{0}(\OS)|=|\OS|$, expresed by $K_{\xi_{0}}\overset{!}{=}0$.
We define
\[
\bar{H}^{1}(\OS):=\left\{ \xi\in H^{1}(\OS):K_{\xi}=0\right\} ,
\]
and, accordingly, 
\[
\Xs:=\left\{ (\xi,\zeta,u)\in\XsK:K_{\xi}=0\right\} .
\]

\begin{lem}
\label{lem:Kxi}The space $\Xs$ is a closed subspace of $\XsK$ that
is invariant under the semigroup. We have $\Xs=\XsK\setminus\mathrm{ker}\AsK$,
where $\mathrm{ker}\AsK$ is one-dimensional, given by 
\[
(\xi,\eta,u)\in\mathrm{ker}\AsK\Leftrightarrow\xi=\kappa\varphi,\eta=0,u=0,
\]
with parameter $\kappa\in\mathbb{R}$ and $\varphi\in H^{1}(\OS)$
a weak solution of the Neumann problem 
\begin{align}
\begin{cases}
\begin{array}{rcll}
\Div(L(\varphi)) & = & \varphi & \text{ in }\OS,\\
L(\varphi)n & = & n & \text{ on }\dOS.
\end{array}\end{cases}\label{eq:xi_stat_system-1}
\end{align}
In particular, 
\begin{equation}
\xi\in H^{1}(\OS)\Leftrightarrow\xi-\kappa_{0}\varphi\in\bar{H}^{1}(\OS),\kappa_{0}=\frac{K_{\xi}}{\Vert\varphi\Vert_{H^{1}(\OS)}}.\label{eq:Kxigone}
\end{equation}
 
\end{lem}

\begin{proof}
Clearly, $\Xs$ is closed in $\XsK$.The characterization of the kernel
of $\AsK$ and thus the invariance is shown in \cite[Theorem 3.2]{4AT2009,8AT2009}.
In particular, this shows that $\AsK$ is not invertible on $\XsK$,
so the (full) semigroup cannot be strongly stable. The characterization
\eqref{eq:Kxigone} can be verified by a short calculation using the
Gauss theorem and \eqref{eq:xi_stat_system-1},
\[
\int_{\dOS}\varphi\cdot n=\int_{\dOS}\varphi\cdot\sigma(\varphi)n=\int_{\OS}\Div(L(\varphi))\cdot\varphi+L(\varphi):\nabla\varphi=\Vert\varphi\Vert_{H^{1}(\OS)}^{2}.
\]
 
\end{proof}
In the following, we may thus also consider the strongly continuous
semigroup of contractions $(\Scs(t))_{t>0}$, which is just the restriction
of $(\ScsK(t))_{t>0}$ to $\Xs$ with generator $\As\colon D(\As)\cap\Xs\subset\Xs\to\Xs$
\cite[I.5.12]{EngelNagel2000}. We need some more properties of $D(\As)$.
The existence of the semigroup on $\XsK$ was proved by the Lumer-Phillips
theorem \cite{LumerPhillips1961}, using that the operator $\AsK$
is dissipative. The maximality corresponds to the identification of
the domain 
\begin{equation}
D(\AsK)\subseteq H^{1}(\OS)\times H^{1}(\OS)\times H^{1}(\OF)\cap\Ls,\label{eq:DA}
\end{equation}
taking the pressure and the boundary conditions into account. The
details are given in \cite[(A.1) - (A.5)]{AT2}. Here, in addition
to the regularity in \eqref{eq:DA}, it is also important that for
all $(\xi,\zeta,u)\in D(\AsK)$, 
\begin{equation}
\Div L(\xi)\in L^{2}(\OS),\quad\text{and }L(\xi)|_{\dOS}n\in H^{-1/2}(\dOS),\label{eq:DA2}
\end{equation}
and that the resolvent of $\AsK$ is not compact, due to the fact
that there is no gain in regularity for $\xi$, \cite[Theorem 2.2]{AT2}.
Further information on the spectrum of $\As$ is given in Theorem
\ref{thm:spectrum}.

\subsection{\label{subsec:Definition-of-Pressure}Definition of \emph{pressure
waves} and \emph{bad domains}}

When considering the fluid part of system \eqref{eq:linsys}, it is
natural to conjecture that $u(t)\overset{t\to\infty}{\to}0$ and that,
as a consequence, when looking just at the first line, 
\[
p(t)\overset{t\to\infty}{\to}q(t)=\frac{1}{|\OF|}\int_{\OF}p(t,x)\,\mathrm{d}x,
\]
a spatially constant function. These properties were shown rigorously
for strong solutions of the corresponding non-linear system \eqref{eq:nonlin_system}.
Also the existence of strong solutions of the linear system \eqref{eq:linsys},
\cite[Theorem 2]{DisserLuckas2025}, and small-data global existence
of strong solutions for \eqref{eq:nonlin_system} is established,
cf.~Subsection~\ref{subsec:resultsluckas}. This convergence of
$u$ to rest also follows from the strong stability of the semigroup
proved in \cite{4AT2009}, but only under additional geometric assumptions
that we discuss also in this section. The vanishing of the fluid velocity
indicates that it is natural to consider solutions $(\eta,q)$ of
the system
\begin{align}
\begin{split}\begin{cases}
\begin{array}{rcll}
\ddot{\eta}-\Div(L(\eta))+\eta & = & 0 & \text{in }(0,T)\times\OS,\\
\eta & = & 0 & \text{on }(0,T)\times\dOS,\\
L(\eta)n & = & qn & \text{on }(0,T)\times\dOS,
\end{array}\end{cases}\end{split}
\label{eq:xi_system_A}
\end{align}
with scalar $q(t)\in\mathbb{R}$ as the candidate limit dynamics for
system \eqref{eq:linsys}. Note the presence of two boundary conditions,
but no initial condition. We call non-trivial solutions $(\eta,q)\neq(0,0)$
of \ref{eq:xi_system_A}\emph{ pressure waves}. Sometimes the term
is used for just the displacement component $\eta\neq0$. Pressure
waves can be characterized by the corresponding overdetermined eigenvalue
problem 
\begin{align}
\begin{split}\begin{cases}
\begin{array}{rcll}
-\Div(L(\psi)) & = & \mu\psi & \text{in }\OS,\\
\psi & = & 0 & \text{on }\dOS,\\
L(\psi)n & = & qn & \text{on }\dOS.
\end{array}\end{cases}\end{split}
\label{eq:xi_overdet_EP-1}
\end{align}
More precisely, by definition, strong solutions $\psi$ of \eqref{eq:xi_overdet_EP-1}
are contained in the set of eigenfunctions of the Dirichlet-Lamé operator
\begin{align*}
\mathcal{L}(\psi) & =-\Div(L(\psi))\qquad\text{with domain}\\
D(\mathcal{L}) & :=\H1_{0}^{1}(\OS)\cap\H1^{2}(\OS)\subset\Lp^{2}(\OS).
\end{align*}
This operator is self-adjoint positive definite and has compact resolvent.
It admits countably many eigenvalues $(\mu_{k})\subset(0,\infty)$
of finite multiplicity that can only cluster at infinity. The corresponding
eigenfunctions $(\tilde{\psi}_{n})\subset D(\mathcal{L})$ (in fact,
they are smooth), with indices $n\in\mathbb{N}$ solving 
\begin{align*}
%\label{eq:xi_{E}P}
\begin{split}\begin{cases}
\begin{array}{rcll}
\mathcal{L}(\psi_{n}) & = & \mu_{n}\psi_{n} & \text{in }\OS,\\
\psi_{n} & = & 0 & \text{on }\dOS,
\end{array}\end{cases}\end{split}
\end{align*}
can be chosen to form an orthonormal basis $(\tilde{\psi}_{n})$ of
$L^{2}(\OS)$. Then the sequence of eigenfunctions $(\psi_{n}:=\frac{1}{\sqrt{1+\mu_{n}}}\tilde{\psi}_{n})$
is an orthonormal basis of the Sobolev space $H_{0}^{1}(\OS)$ with
zero boundary trace and inner product as in \eqref{eq:innerproduct}.
Clearly, $H_{0}^{1}(\OS)\subset\bar{H}^{1}(\OS)$. The set of pressure
waves is thus given as 
\begin{align}
\omega_{T}= & \biggl\{\eta:\eta(t,y)=\sum_{k\in K}(a_{k}\sin(\sqrt{1+\mu_{k}}t)+b_{k}\cos(\sqrt{1+\mu_{k}}t))\psi_{k}(y):(a_{k}),(b_{k})\in\mathbb{R}\biggr\}\label{eq:charomega}\\
 & \cap C([0,T];H_{0}^{1}(\OS))\cap C^{1}([0,T];L^{2}(\OS)),\nonumber 
\end{align}
where %
\begin{comment}
\begin{proof}
Let $(\eta,qn)\in\omega_{T}$. Since $(\psi_{k})$ forms an orthonormal
base of $\Lp^{2}(\OS)$, there exist functions $(c_{k})\subset\C^{0}(0,T)$
such that 
\begin{align*}
\eta(t,y)=\sum_{k=0}^{\infty}c_{k}(t)\psi_{k}(y)\hspace{0.2cm}\text{for all }t\in(0,T)\text{ and a.e.\@ }y\in\OS.
\end{align*}
For any $\tilde{k}\in\mathbb{N}$, integration by parts implies 
\[
\langle\ddot{\eta},\psi_{\tilde{k}}\rangle_{\H1^{-1}(\OS),\H1_{0}^{1}(\OS)}=-\mu_{\tilde{k}}c_{\tilde{k}}(t)
\]
and hence 
\begin{align*}
\ddot{\eta}(t,y)=\Div(\Sigma(\eta))=\sum_{k=0}^{\infty}-\mu_{k}c_{k}(t)\psi_{k}(y).
\end{align*}
This also implies that $\ddot{c}_{k}=-\mu_{k}c_{k}$ (in a weak sense)
and thus 
\begin{align*}
c_{k}(t)=a_{k}\sin(\sqrt{\mu_{k}}t)+b_{k}\cos(\sqrt{\mu_{k}}t)
\end{align*}
for some $a_{k},b_{k}\in\mathbb{R}$. If $\OS$ is a \emph{good} domain,
then by definition this system has the only solution $\psi_{k}=0$,
$q_{k}=0$ and we can conclude that $\omega_{T}=\{(0,0)\}$. For a
bad domain $\OS$, we obtain that 
\begin{align}
\omega_{T} & =\left\{ \sum_{k\in I}(a_{k}\sin(\sqrt{\mu_{k}}t)+b_{k}\cos(\sqrt{\mu_{k}}t))(\psi_{k}(y),q_{k}):(\psi_{k},q_{k})\text{ solve }(\ref{eq:xi_overdet_EP}),\,(a_{k}),(b_{k})\in\mathbb{R}\right\} \nonumber \\
 & \hspace{0.5cm}\cap(\C^{0}(\H1_{0}^{1}(\OS))\cap\C^{1}(\Lp^{2}(\OS))\cap\C^{2}(\H1^{-1}(\OS))\cap\H1^{2}(\H1^{1}(\OS)^{*})))\times\C^{0}(0,T)\\
 & \neq\{(0,0)\}.\qedhere\nonumber 
\end{align}
\end{proof}
\end{comment}
\[
K:=\{k\in\mathbb{N}\colon\exists(\psi_{k},q_{k})\neq(0,0)\text{ solution of }\eqref{eq:xi_overdet_EP-1}_{\mu_{k}}\}
\]
is the set of indices $k\in\mathbb{N}$ of (non-trivial) eigenfunctions
$\psi_{k}$ that solve \eqref{eq:xi_overdet_EP-1}. Here, the regularity
imposed on $\omega_{T}$ corresponds to the regularity of the semigroup
mild solutions for \eqref{eq:linsys}. Clearly $K\subseteq\mathbb{N}$,
and the structure of this set and of the corresponding eigenvalues
and eigenfunctions depends on the geometry of $\OS$. Following the
work of Avalos and Triggiani \cite{5AT,4AT2009,AT2}, we use the following
definition. 
\begin{defn}
The domain $\OS$ is called
\end{defn}

\begin{itemize}
\item a \emph{good} domain, if $\psi=0$, $q=0$ is the only solution of
(\ref{eq:xi_overdet_EP-1}), i.e. $K=\emptyset$. 
\item a \emph{bad} domain, if (\ref{eq:xi_overdet_EP-1}) admits a non-zero
solution $(\psi_{k},q_{k})$ for some $\mu_{k}>0$, i.e. $K\neq\emptyset$.
\end{itemize}
\begin{prop}
\label{prop:ATgooddomains}It was shown in \cite{5AT,AT2} for a similar
wave-type system that every domain which is partially flat, partially
spherical, partially elliptic, partially hyperbolic or partially parabolic
is a \emph{good} domain, and this result was transferred to the Lamé
system in \cite[Remark~1.1]{ATboundary}. 
\end{prop}

\begin{example}
\label{exa:ball}A known \emph{bad} domain is the ball \cite{5AT}.
For $\OS=B_{r}(0)$, examples of non-zero solutions to (\ref{eq:xi_overdet_EP-1})
are given by 
\begin{align*}
\psi_{k}(y) & :=\left(\frac{r^{2}\sin\left(\frac{r_{k}}{r}\vert y\vert\right)}{r_{k}^{2}\vert y\vert^{3}}-\frac{r\cos\left(\frac{r_{k}}{r}\vert y\vert\right)}{r_{k}\vert y\vert^{2}}\right)y,\\
q_{k} & :=(2\lambda_{1}+\lambda_{2})\sin(r_{k}),
\end{align*}
with eigenvalues 
\begin{align*}
\mu_{k}=\frac{(2\lambda_{0}+\lambda_{1})r_{k}^{2}}{r^{2}},
\end{align*}
where $r_{k}\in(0,\infty)$ is the $k$-th positive root of the spherical
Bessel function 
\begin{align*}
j(r)=\frac{\sin(r)}{r^{2}}-\frac{\cos(r)}{r}.
\end{align*}

In Section \ref{sec:Schiffer}, we will further discuss the following 
\end{example}

\begin{conjecture}
\label{conj:Baddomains}The balls $\OS=B_{r}(x_{0})$, for any $r>0,x_{0}\in\mathbb{R}^{3}$,
are the only bad domains. 
\end{conjecture}

\subsection{Long-time behaviour of the semigroup on good domains}

By characterizing the spectrum of $\As$ on the imaginary axis and
using the criterion of Arendt-Batty and Lyubich-Vu, \cite{ArendtBatty1988,LyubichPhong1988},
Avalos and Triggiani \cite[Theorem 2.3, Theorem 3.2,  Property (6.1)]{8AT2009}
showed the following:
\begin{thm}
\label{thm:spectrum}The point spectrum $\sigma_{p}(\As)$ of $\As$
on the imaginary axis is given exactly by the at most countably many
eigenvalues of the overdetermined problem \eqref{eq:xi_overdet_EP-1},
\[
\sigma_{p}(\As)\cap i\mathbb{R}=\left\{ \pm i\sqrt{1+\mu_{k}}:k\in K\right\} .
\]
Otherwise, there are no purely imaginary spectral points of $\As$,
\[
i\mathbb{R}\cap\left\{ \sigma(\As)\setminus\sigma_{p}(\As)\right\} =\emptyset.
\]
 If $\OS$ is a good domain, then $(\Scs(t))_{t>0}$ is strongly stable
on $\Xs$. 
\end{thm}

\section{\label{sec:Invariance}Invariance of pressure waves and strong stability
on bad domains}

Here, we want to extend the characterization of the long-time dynamics
of the semigroup to bad domains. Let $\OS$ be a bad domain and recall
the definition 
\[
K=\{k\in\mathbb{N}:\exists(\psi_{k},q_{k})\neq(0,0)\text{ solution of }\eqref{eq:xi_overdet_EP-1}_{\mu_{l}}\}
\]
of the set of indices for which the Dirichlet-Lamé eigenfunctions
provide a non-trivial solution of \eqref{eq:xi_overdet_EP-1}. The
discussion in Section \ref{sec:Schiffer} below will show that $K\neq\emptyset$
may be infinite (if $\OS$ is a ball) or finite (if $\OS$ is a different
bad domain). Let 
\begin{align*}
E & :=\left\{ \xi\in\bar{H}^{1}(\OS):(\xi,\psi_{k})_{H^{1}(\OS)}=0\text{ for all }k\in K\right\} ,\text{ and }\\
\tilde{E} & :=\left\{ \zeta\in L^{2}(\OS):\int_{\OS}\zeta\tilde{\psi}_{k}=0\text{ for all }k\in K\right\} 
\end{align*}
be the spaces of functions perpendicular to the eigenfunctions $\psi_{k},k\in K$
and 

\begin{align*}
W & :=\left\{ \xi\in\bar{H}^{1}(\OS):\xi\in\overline{\mathrm{span}\{\psi_{k}:k\in K\}}^{H^{1}(\OS)}\right\} ,\text{and }\\
\tilde{W} & :=\left\{ \zeta\in L^{2}(\OS):\zeta\in\overline{\mathrm{span}\{\tilde{\psi}_{k}:k\in K\}}^{L^{2}(\OS)}\right\} 
\end{align*}
their spans. For all $\xi\in H^{1}(\OS),\zeta\in L^{2}(\OS)$, and
$k\in K$, define the expressions 
\begin{equation}
\xi_{k}^{H}:=(\xi,\psi_{k})_{H^{1}(\OS)},\qquad\zeta_{k}^{L}:=(\zeta,\tilde{\psi}_{k})_{L^{2}(\OS)},\label{eq:coeffk-1}
\end{equation}
and the projections 
\[
\mathcal{P}_{W}\colon\begin{array}{rcl}
\bar{H}^{1}(\OS) & \to & W\subseteq\bar{H}^{1}(\OS)\\
\xi & \mapsto & \Sigma_{k\in K}\xi_{k}^{H}\psi_{k}
\end{array},
\]
and 
\[
\mathcal{P}_{\tilde{W}}\colon\begin{array}{rcl}
L^{2}(\OS) & \to & \tilde{W}\subseteq L^{2}(\OS)\\
\zeta & \mapsto & \Sigma_{k\in K}\zeta_{k}^{L}\tilde{\psi}_{k}
\end{array},
\]
as well as $\mathcal{P}_{E}:=\mathrm{Id}-\mathcal{P}_{W}$ and $\mathcal{P}_{\tilde{E}}:=\mathrm{Id}-\mathcal{P}_{\tilde{W}}$.
Due to the continuity of the inner products and the orthogonality
of the bases $(\psi_{n}),(\tilde{\psi_{n}})$, $\mathcal{P}_{W},\mathcal{P}_{\tilde{W}}$
and $\mathcal{P}_{E},\mathcal{P}_{\tilde{E}}$ are bounded and $E,\tilde{E}$
and $W,\tilde{W}$ are closed subspaces of $H^{1}(\OS)$ and $L^{2}(\OS)$,
respectively. We set
\[
\Es:=E\times\tilde{E}\times\Ls\subseteq\Xs\qquad\text{and }\Ws:=W\times\tilde{W}\times\{0\}\subseteq\Xs
\]
and define the corresponding bounded projections 
\[
\mathcal{P}_{\Ws}:=(\mathcal{P}_{W},\mathcal{P}_{\tilde{W}},0)\colon\Xs\to\Ws;\qquad\mathcal{P}_{\Es}:=(\mathcal{P}_{E},\mathcal{P}_{\tilde{E}},\mathrm{Id}_{L})\colon\Xs\to\Es.
\]
We have thus obtained the orthogonal decomposition 
\begin{equation}
\Es\oplus\Ws=\Xs.\label{eq:decomp}
\end{equation}

\begin{lem}
\label{lem:invariance}The projections $\mathcal{P}_{\Ws},\mathcal{P}_{\Es}$
commute with the semigroup $(\Scs(t))_{t>0}$. In particular, the
spaces $\Es$ and $\Ws$ are invariant. 
\end{lem}

\begin{proof}
For all $(\xi_{0},\xi_{1},u_{0})\in\Xs$, the map $t\mapsto\Scs(t)(\xi_{0},\xi_{1},u_{0})$
gives a mild solution of \eqref{eq:linsys}, in the sense that for
all $t\in(0,T)$,
\begin{equation}
\int_{0}^{t}\left(\begin{array}{c}
\xi(s)\\
\zeta(s)\\
u(s)
\end{array}\right)\,\mathrm{d}s\in D(\As),\quad\label{eq:mildeq}
\end{equation}
and 
\begin{equation}
\left(\begin{array}{c}
\xi(t)\\
\zeta(t)\\
u(t)
\end{array}\right)=\left(\begin{array}{c}
\xi_{0}\\
\xi_{1}\\
u_{0}
\end{array}\right)+\As\int_{0}^{t}\left(\begin{array}{c}
\xi(s)\\
\zeta(s)\\
u(s)
\end{array}\right)\,\mathrm{d}s.\label{eq:mildeq2}
\end{equation}
For all $k\in K$, take the $H^{1}$- and $L^{2}$- inner products
with $\psi_{k},\tilde{\psi}_{k}$ in the first two equations. We can
use that 
\begin{equation}
\xi^{*}(t):=\int_{0}^{t}\xi(s)\,\mathrm{d}s\in\bar{H}^{1}(\OS),\quad\zeta^{*}(t):=\int_{0}^{t}\zeta(s)\,\mathrm{d}s\in\bar{H}^{1}(\OS),\label{eq:invarianceK}
\end{equation}
where the first inclusion is just the invariance from Lemma \ref{lem:Kxi},
and the second inclusion is due to the corresponding property $R(\As)|_{\zeta}\subseteq\bar{H}^{1}(\OS)$,
shown in \cite[(2.15)]{AT2}. Due to $\xi^{*}(t)\in D(\As)|_{\xi}$,
by \eqref{eq:DA2}, $\Div\sigma(\xi^{*})(t)\in L^{2}(\OS)$ and $\sigma(\xi^{*}(t))n\in H^{-1/2}(\dOS)$
are well-defined. We have 
\[
\int_{0}^{t}\xi_{k}^{H}(s)\,\mathrm{d}s=(\xi^{*}(t),\psi_{k})_{H^{1}(\OS)}
\]
and $\int_{0}^{t}\zeta_{k}^{L}(s)\,\mathrm{d}s=(\zeta^{*}(t),\tilde{\psi}_{k})_{L^{2}(\OS)}$
for all $t\geq0$. Due to these considerations, due to \eqref{eq:xi_overdet_EP-1},
by the Gauss Theorem and using \eqref{eq:invarianceK}, we obtain
\begin{align}
\xi_{k}^{H}(t) & =\xi_{0,k}^{H}+(\zeta^{*}(t),\psi_{k})_{H^{1}(\OS)}\label{eq:odeXik}\\
 & =\xi_{0,k}^{H}-\int_{\OS}\zeta^{*}(t)\cdot\Div L(\psi_{k})+q_{K}\int_{\dOS}\zeta^{*}(t)\cdot n+(\zeta^{*}(t),\psi_{k})_{L^{2}(\OS)}\nonumber \\
 & =\xi_{0,k}^{H}+\sqrt{1+\mu_{k}}(\zeta^{*}(t),\tilde{\psi}_{k})_{L^{2}(\OS)}=\xi_{0,k}^{H}+\sqrt{1+\mu_{k}}\int_{0}^{t}\zeta_{k}^{L}(s)\,\mathrm{d}s,\nonumber 
\end{align}
and 
\begin{align}
\zeta_{k}^{L}(t) & =\xi_{1,k}^{L}+(\Div L(\xi^{*}(t)),\tilde{\psi}_{k})_{L^{2}(\OS)}-(\xi^{*}(t),\tilde{\psi}_{k})_{L^{2}(\OS)}\label{eq:odezetak}\\
 & =\xi_{1,k}^{L}-\int_{\OS}L(\xi^{*}(t)):\nabla\tilde{\psi}_{k}-\int_{0}^{t}\xi_{k}^{L}(s)\,\mathrm{d}s=\xi_{1,k}^{L}-\sqrt{1+\mu_{k}}\int_{0}^{t}\xi_{k}^{H}(s)\,\mathrm{d}s.\nonumber 
\end{align}
The unique solution of this integrated system of ODEs is given by
\[
\left(\begin{array}{c}
\xi_{k}^{H}(t)\\
\zeta_{k}^{L}(t)
\end{array}\right)=e^{t\sqrt{1+\mu_{k}}\left(\begin{array}{cc}
0 & 1\\
-1 & 0
\end{array}\right)}\left(\begin{array}{c}
\xi_{0,k}^{H}\\
\xi_{1,k}^{L}
\end{array}\right).
\]
This is also exactly the system solved by the coefficients of any
pressure wave $\eta(t)=\Sigma_{k\in K}\xi_{1,k}^{L}\sin(\sqrt{1+\mu_{k}}t)+\xi_{0,k}^{H}\cos(\sqrt{1+\mu_{k}}t)\psi_{k}$
given as in the characterization \eqref{eq:charomega}, eminating
from initial data $\eta_{0}=\Sigma_{k\in K}\xi_{0,k}^{H}\psi_{k}\in H_{0}^{1}(\OS),\eta_{1}=\Sigma_{k}\xi_{1,k}^{L}\tilde{\psi}_{k}\in L^{2}(\OS)$.
Adding a fluid at rest to the pressure wave, $u(t)=0$, provides the
unique mild solution of \eqref{eq:linsys}. This shows that for any
initial data $\mathbf{x}_{0}=(\xi_{0},\xi_{1},u_{0})\in\Xs$, 
\[
\mathcal{P}_{\Ws}\left(\Scs(t)\mathbf{x}_{0}\right)=\Scs(t)\left(\mathcal{P}_{\Ws}\mathbf{x}_{0}\right).
\]
By linearity and using the decomposition $\Xs=\Es\oplus\Ws$, this
also implies 
\[
\mathcal{P}_{\Es}\left(\Scs(t)\mathbf{x}_{0}\right)=\Scs(t)\left(\mathcal{P}_{\Es}\mathbf{x}_{0}+\mathcal{P}_{\Ws}\mathbf{x}_{0}\right)-\mathcal{P}_{\Ws}\left(\Scs(t)\mathbf{x}_{0}\right)=\Scs(t)\left(\mathcal{P}_{\Es}\mathbf{x}_{0}\right).
\]
The invariances $\Scs(t)(\Es)\subseteq\Es$ and $\Scs(t)(\Ws)\subseteq\Ws$
are a direct consequence. 
\end{proof}
Based on this decomposition lemma, the following main result characterizes
the long-time behaviour of the fluid-elastic semigroup.
\begin{thm}
\label{thm:strongstab}For all Lipschitz domains $\OS$, for all initial
values $\mathbf{x}_{0}=(\xi_{0},\xi_{1},u)\in\XsK$, there is a unique
decomposition 
\[
\left(\begin{array}{c}
\xi_{0}\\
\xi_{1}\\
u_{0}
\end{array}\right)=\left(\begin{array}{c}
\kappa_{0}\varphi\\
0\\
0
\end{array}\right)+\left(\begin{array}{c}
\xi_{0}^{E}\\
\xi_{1}^{\tilde{E}}\\
u_{0}
\end{array}\right)+\left(\begin{array}{c}
\eta_{0}\\
\eta_{1}\\
0
\end{array}\right),
\]
where $\kappa_{0}\in\mathbb{R}$,$\varphi\in H^{1}(\Omega)$ are given
by Lemma \ref{lem:Kxi}, $(\xi_{0}^{E},\xi_{1}^{\tilde{E}},u_{0})\in\Es$
and $(\eta_{0},\eta_{1},0)\in\Ws$. Correspondingly, the mild solution
of \eqref{eq:linsys} decomposes into 
\begin{equation}
\left(\begin{array}{c}
\xi\\
\dot{\xi}\\
u
\end{array}\right)(t)=\ScsK(t)\left(\begin{array}{c}
\xi_{0}\\
\xi_{1}\\
u_{0}
\end{array}\right)=\left(\begin{array}{c}
\kappa_{0}\varphi\\
0\\
0
\end{array}\right)+\Scs(t)|_{\Es}\left(\begin{array}{c}
\xi_{0}^{E}\\
\xi_{1}^{\tilde{E}}\\
u_{0}
\end{array}\right)+\Scs(t)|_{\Ws}\left(\begin{array}{c}
\eta_{0}\\
\eta_{1}\\
0
\end{array}\right),\label{eq:decosol}
\end{equation}
where $(\ScsK(t))_{t>0}$ is a strongly continuous semigroup of contractions
on $\XsK$, and the spaces $\Es$ and $\Ws$ are invariant under the
restriction $(\Scs(t)|_{\Xs})_{t>0}$ of the semigroup to $\Xs$.
The restricition $(\Scs(t)|_{\Es})_{t>0}$ of the semigroup to $\Es$
is strongly stable. In particular 
\[
\Vert u(t)\Vert_{L^{2}(\OS)}\overset{t\to\infty}{\to}0,\qquad\Vert\xi(t)-\kappa_{0}\varphi-\eta(t)\Vert_{H^{1}(\OS)}\overset{t\to\infty}{\to}0,
\]
where $\eta$is the unique pressure wave solution of \eqref{eq:xi_system_A}
eminating from initial data $\eta(0)=\eta_{0},\dot{\eta}(0)=\eta_{1}$.
If $\OS$ is a good domain, then $\Es=\Xs$ and thus $\eta_{0}=\eta_{1}=0$. 
\end{thm}

\begin{proof}
The decomposition of initial data is due to Lemma \ref{lem:Kxi} and
\eqref{eq:decomp}. The fact that $(\ScsK(t))_{t>0}$ is a strongly
continuous semigroup of contractions generated by $\AsK$ was shown
in \cite{5AT,4AT2009}.The decomposition \eqref{eq:decosol} of the
solution follows from Lemmas \ref{lem:Kxi} and \ref{lem:invariance}.
It remains to prove the strong stability of $(\Scs(t)|_{\Es})_{t>0}$.
This is obtained from the Arendt-Batty and Lyubich-Vu criterion \cite{ArendtBatty1988,LyubichPhong1988}:
it is sufficient that no eigenvalue of $\As|_{\Es}:D(\As)\cap\Es\to\Es$
lies on the imaginary axis and that the spectrum $i\mathbb{R}\cap\sigma(\As|_{\Es})$
of $\As|_{\Es}$ on the imaginary axis is countable. Both statements
can be obtained as consequences of the characterization of the spectrum
of $\As$ on $\Xs$ in Theorem \ref{thm:spectrum}. The imaginary
number $i\beta$ is an eigenvalue of $\As$ (on $\Xs$) if and only
if $\beta=\pm\sqrt{1+\mu_{n}}$ for some $k\in K$, $\mu_{k}$ an
eigenvalue of the overdetermined problem \eqref{eq:xi_overdet_EP-1}.
In this case, the corresponding solutions $w_{k}:=(\psi_{k},i\mu_{k}\psi_{k},0)$
span the eigenspace of $\As$, which is thus contained in $\Ws$.
But $w\notin\Es$ for all $0\neq w\in\Ws$, so there is no eigenvalue
of $\As|_{\Es}$ on the imaginary axis. Since $\As$ generates a semigroup
of contractions, its spectrum is contained in the non-positive left
half-plane, so every $\mu\in i\mathbb{R}$ lies in the closure of
the connected component of the resolvent of $\As$ that opens up to
the right, $\rho_{+}(\As)$, of $\As$. By \cite[Corollary 2.16]{EngelNagel2000},
for all $\mu\in\rho_{+}(\As)$, if $\mu\in\sigma(\As|_{\Es})$ then
$\mu\in\sigma(\As)$. By Theorem \ref{thm:spectrum}, $i\mathbb{R}\cap\sigma(A)$
is countable, so $i\mathbb{R}\cap\sigma(\As|_{\Es})$ must be countable
as well and the claim is proved. 
\end{proof}

\section{\label{sec:Schiffer}Bad domains are Schiffer domains and vice versa}

The aim of this section is to obtain a new characterization of bad
domains. In fact, we show that a sufficiently regular non-trivial
solution $\eta$ of \eqref{eq:xi_overdet_EP-1} gives us a non-trivial
solution of the scalar overdetermined Neumann problem
\begin{equation}
\begin{split}\begin{cases}
\begin{array}{rcll}
-\Delta u & = & \mu_{S}u, & \text{in }\OS,\\
\partial_{n}u & = & 0, & \text{on }\dOS,\\
u & = & c, & \text{on }\dOS.
\end{array}\end{cases}\end{split}
\label{eq:schifferProblem}
\end{equation}
The following statement is known as the Schiffer conjecture:
\begin{conjecture*}
\label{claim:schiffer}If \eqref{eq:schifferProblem} admits a non-trivial
solution and $\OS$ is a bounded Lipschitz domain, then $\OS$ is
a ball. 
\end{conjecture*}
We call $\OS$ a \emph{Schiffer domain}, if \eqref{eq:schifferProblem}
admits a solution. The proof of Claim \ref{claim:schiffer} is an
open problem in geometric analysis. Partial results are known. We
focus here on the smooth three-dimensional case and refer to \cite{KawohlLucia2020Schiffer}
and references therein for the two-dimensional case. In any dimension,
the Schiffer Conjecture is essentially equivalent to the Pompeiu problem
\cite{Williams1976SchifferEqtoPompeiu}, which is surveyed in \cite{zalcman1992}.
It is well-known that if $c=0$, then there is only the trivial solution
$u=0$ \cite[Lemma 2.4]{LiuJDE2007Schiffer}. If $\OS$ is a Schiffer
domain, then it is real analytic \cite{Williams1981SchifferAnalyticityfromLipschitz}.
If \eqref{eq:schifferProblem} admits infinitely many eigenvalues
$\mu_{S}$, then $\OS$ must be a ball \cite{BerensteinYang1987SchifferFiniteness}.
Counterexamples for the Schiffer conjecture can be given on manifolds
\cite{FallMinlendWeth2023Schiffer}.
\begin{claim}
Here, we prove the following connection of the Schiffer conjecture
to the problem of characterizing bad domains. 
\end{claim}

\begin{thm}
\label{thm:Schiffer}Let $\OS$ be a $C^{2,1}$-domain. Then $\OS$
is a bad domain, if and only if it is a Schiffer domain. 
\end{thm}

\begin{proof}
The theorem is due to the following two (simple) explicit constructions:
\begin{enumerate}
\item \label{enu:div}If $\xi$ solves \eqref{eq:xi_overdet_EP-1}, then
$v=\Div\xi$ solves \eqref{eq:schifferProblem} with $\mu_{S}=\frac{\mu}{\lambda}$
and $c=\frac{q}{\lambda}$, where $\lambda=\lambda_{0}+\lambda_{1}$. 
\item \label{enu:grad}If $u$ solves \eqref{eq:schifferProblem}, then
$V=\nabla u$ solves \eqref{eq:xi_overdet_EP-1} with $\mu=\lambda\mu_{S}$
and $q=-\lambda\mu_{S}c$. 
\end{enumerate}
We first show \eqref{enu:div}. A direct calculation establishes the
PDE: 
\[
\frac{\mu}{\lambda}u=\frac{\mu}{\lambda}\Div\xi=-\frac{1}{\lambda}\Div\Div L(\xi)=-\Delta\Div\xi=-\Delta u.
\]
We need to consider the boundary conditions. For the following calculations,
there are several choices of the presentation, e.g. in the language
of differential geometry or by using local charts. We use a simple
slightly hybrid version. Let $x^{0}\in\dOS$ be given. Possibly after
a rigid motion, withouth loss of generality, we can assume that $x^{0}=0,$
that $n(x^{0})=(0,0,1)^{T}=e^{3}$ is the unit outer normal vector
of $\OS$ at this point, that $e^{1}:=(1,0,0)^{T},e^{2}:=(0,1,0)^{T}$
are tangent vectors, and that in a neighbourhood $U$ of $0\in\dOS$,
$\dOS$ is given as the graph of a $C^{2,\beta}$-function, 
\[
x\in\dOS\cap U\Leftrightarrow x=\left(\begin{array}{c}
x_{1}\\
x_{2}\\
h(x_{1},x_{2})
\end{array}\right),(x_{1},x_{2})^{T}\in V:=h^{-1}(U\cap\dOS).
\]
This means that for any $x=(x_{1},x_{2},x_{3})=(\bar{x},x_{3})\in U\cap\dOS$,
normal and tangent vectors (not necesarrily of unit length) are given
by 
\[
n(x):=\left(\begin{array}{c}
-\partial_{1}h(\bar{x})\\
-\partial_{2}h(\bar{x})\\
1
\end{array}\right),\qquad\tau^{1}(x):=\left(\begin{array}{c}
1\\
0\\
\partial_{1}h(\bar{x})
\end{array}\right),\qquad\tau^{2}(x):=\left(\begin{array}{c}
0\\
1\\
\partial_{2}h(\bar{x})
\end{array}\right),
\]
and that $\partial_{1}h(0,0)=\partial_{2}h(0,0)=0$ hold, cf.~\cite[Chapter I.2]{wloka}.
At any $x\in\dOS\cap U$, for all $k=1,2,3$, the Dirichlet boundary
condition in \eqref{eq:xi_overdet_EP-1} implies 
\begin{equation}
0=\partial_{\tau^{j}(x)}\xi_{k}(x)=\overset{x\to0}{\to}\partial_{j}\xi_{k}(0),\qquad j=1,2.\label{eq:DBC1}
\end{equation}
Using this, at $x=0$, the Neumann boundary condition in \eqref{eq:xi_overdet_EP-1}
provides
\begin{align}
0 & =\lambda_{0}D(\xi)_{13}(0)=\frac{\lambda_{0}}{2}\partial_{3}\xi_{1}(0),\nonumber \\
0 & =\lambda_{0}D(\xi)_{23}(0)=\frac{\lambda_{0}}{2}\partial_{3}\xi_{2}(0),\nonumber \\
q & =(L(\xi))_{33}(0)=\lambda\partial_{3}\xi_{3}(0).\label{eq:NBC1}
\end{align}
In particular, by summing up, we obtain the Dirichlet condition for
$v$, $v(0)=\Div\xi(0)=\frac{q}{\lambda}$. It remains to prove the
Neumann boundary condition $\partial_{3}v(0)=0$. To improve the notation
in the remaining calculations, for all $x\in\OS\cap U,$we define
\begin{align*}
A(x) & :=\partial_{3}^{2}\xi_{3}(x),\\
B(x) & :=\partial_{1}^{2}\xi_{3}(x)+\partial_{2}^{2}\xi_{3}(x),\\
C(x) & :=\partial_{3}\partial_{1}\xi_{1}(x)+\partial_{3}\partial_{2}\xi_{2}(x),
\end{align*}
and sometimes omit the argument $(x)$. Tangentially differentiating
the Dirichlet boundary condition \eqref{eq:DBC1} a second time gives,
for all $x\in\dOS\cap U$, 
\begin{align*}
0 & =\partial_{\tau^{j}(x)}^{2}\xi_{3}(x)=\Sigma_{k,l=1}^{3}\tau_{l}^{j}\partial_{l}\tau_{k}^{j}\partial_{k}\xi_{3}+\tau_{l}^{j}\tau_{k}^{j}\partial_{k}\partial_{l}\xi_{3}\\
 & \overset{x=0}{=}\partial_{3}\xi_{3}(0)\partial_{j}^{2}h(0)+\partial_{j}^{2}\xi_{3}(0)\overset{\eqref{eq:NBC1}}{=}\frac{q}{\lambda}\partial_{j}^{2}h(0)+\partial_{j}^{2}\xi_{3}(0),\qquad j=1,2,
\end{align*}
so 
\begin{equation}
B(0)=-\frac{q}{\lambda}H(0),\label{eq:SH}
\end{equation}
where $H(0)=\partial_{1}^{2}h(0)+\partial_{2}^{2}h(0)$ denotes the
mean curvature of $\dOS$ at $0$. Also for all $x\in\dOS\cap U$,
the Neumann boundary condition for $\xi$ gives 
\[
0=\tau^{j}(x)\cdot qn(x)=\tau^{j}(x)\cdot(L(\xi)n)(x),\qquad j=1,2.
\]
From differentiating these quantities, we get, for both $j=1,2$,
\begin{align*}
0 & =\partial_{\tau^{j}(x)}\left(\tau^{j}(x)\cdot qn(x)\right)=\Sigma_{k=1}^{3}\partial_{j}^{2}h(x)\sigma_{3k}n_{k}+\tau^{j}(x)\left(\partial_{j}\sigma\right)n+\tau^{j}\sigma\left(\partial_{j}n(x)\right)\\
 & \overset{x=0}{=}\partial_{j}^{2}h(0)\sigma_{33}(0)+\partial_{j}\sigma_{j3}(0)+\Sigma_{l=1}^{2}\sigma_{jl}(0)\left(-\partial_{j}\partial_{l}h(0)\right)\\
 & \overset{\eqref{eq:DBC1},\eqref{eq:NBC1}}{=}\partial_{j}^{2}h(0)q+\frac{\lambda_{0}}{2}\left(\partial_{j}^{2}\xi_{3}+\partial_{3}\partial_{j}\xi_{j}\right)-\partial_{j}^{2}h(0)\lambda_{1}\frac{q}{\lambda}.
\end{align*}
By summing up, we have 
\begin{equation}
\frac{\lambda_{0}}{2}(C+B)(0)=-\frac{\lambda_{0}}{\lambda}qH(0).\label{eq:ST}
\end{equation}
Using \eqref{eq:SH}, this gives
\begin{equation}
B(0)=C(0).\label{eq:ST2}
\end{equation}
For all $x\in\OS\cap U$, the PDE for the third component $\xi_{3}$
reads as 
\begin{equation}
\lambda(A+C)(x)-\frac{\lambda_{0}}{2}(C-B)(x)=-\mu\xi_{3}(x).\label{eq:PDE3}
\end{equation}
By combining the previous identities, we thus obtain 
\[
\lambda\partial_{3}v(x)=\lambda(A+C)(x)\overset{\eqref{eq:PDE3}}{=}-\mu\xi_{3}(x)+\frac{\lambda_{0}}{2}(C-B)(x)\overset{x\to0,\eqref{eq:ST2}}{\to}-\mu\xi_{3}(0)=0.
\]
Now let $u$ be a given solution of \eqref{eq:schifferProblem}. Then,
clearly, $-\Div L(\nabla u)=\lambda\mu_{S}\nabla u.$ We prove the
boundary conditions for the converse statement: As before, fix $0=x^{0}\in\dOS$,
and let $\dOS\cap U$ be given by the graph of $h$. Due to the constant
Dirichlet boundary condition $u|_{\dOS}=c$, the tangential derivatives
$\partial_{\tau^{1}(x)}u(x)=\partial_{\tau^{2}(x)}u(x)=0$ vanish.
At $x=0,$ $\partial_{1}u(0)=\partial_{\tau^{1}(0)}u(0)$, $\partial_{2}u(0)=\partial_{\tau^{2}(0)}u(0)$
and $\partial_{3}u(0)=\partial_{n}u(0)=0$ is the given Neumann boundary
condition, so $V(0)=\nabla u(0)=0$. It remains to prove the three
components of the Neumann boundary condition, namely that for some
$q\in\mathbb{R}$, 
\begin{align}
\left(L(V)n\right)_{1}(0)=\lambda_{0}\partial_{1}\partial_{3}u(0) & \overset{!}{=}0,\nonumber \\
\left(L(V)n\right)_{2}(0)=\lambda_{0}\partial_{2}\partial_{3}u(0) & \overset{!}{=}0,\nonumber \\
\left(L(V)n\right)_{3}(0)=\lambda_{0}\partial_{3}^{2}u(0)+\lambda_{1}\partial_{3}(\Div u)(0) & \overset{!}{=}q.\label{eq:NBtoshow}
\end{align}
By differentiating the homogenous Neumann boundary condition for $u$
in tangential directions, and using that $\nabla u(0)=0$, we see
that for all $x\in\dOS\cap U$, 
\begin{equation}
0=\partial_{\tau^{j}(x)}\partial_{n(x)}u(x)=\Sigma_{k,l=1}^{3}\tau_{k}^{j}\left(\left(\partial_{k}\partial_{l}u\right)n_{l}+\left(\partial_{l}u\right)\left(\partial_{k}n_{l}\right)\right)\overset{x=0}{=}\partial_{j}\partial_{3}u(0),\qquad j=1,2.\label{eq:NBC3}
\end{equation}
This takes care of the first two Neumann boundary conditions for $V$.
The constant Dirichlet boundary condition for $u$ implies 
\[
0=\partial_{\tau^{j}(x)}u(x)=\partial_{j}u(x)+\partial_{3}u(x)(\partial_{j}h(x))),\qquad j=1,2.
\]
By differentiating these quantities, and using the Neumann condition
$\partial_{3}u(0)=0$, we get
\begin{equation}
0=\partial_{j}^{2}u(x)+\partial_{j}\partial_{3}u(x)(\partial_{j}h(x))+\partial_{3}u(x)(\partial_{j}^{2}h(x))\overset{x=0}{=}\partial_{j}^{2}u(0),\qquad j=1,2.\label{eq:DBC3}
\end{equation}
For the third line in \eqref{eq:NBtoshow}, by using the PDE, we thus
obtain
\begin{align*}
\left(L(V)n\right)_{3}(0) & \overset{\eqref{eq:NBC3}}{=}\lambda\partial_{3}^{2}u(0)\overset{\eqref{eq:DBC3}}{=}\lambda\Delta u(0)\\
 & =\lim_{x\to0}\lambda\Delta u(x)=\lim_{x\to0}-\lambda\mu_{S}u(x)=-\lambda\mu_{S}c=:q,
\end{align*}
and Theorem \ref{thm:Schiffer} is proved. 
\end{proof}
\begin{rem}
The construction in the proof of the previous theorem shows that if
$\xi$ is a sufficiently smooth solution of \eqref{eq:xi_overdet_EP-1},
then $\Xi=-\frac{1}{\mu}\nabla\Div\xi$ is also a solution of \eqref{eq:xi_overdet_EP-1}
with the same eigenvalue and the same constant in the Neumann condition.
This means that if there exists a non-trivial solution at all, then
there must also be a non-trivial gradient-of-a-potential solution.
To prove Conjecture \ref{conj:Baddomains} (and the Schiffer conjecture),
it would thus be sufficient to show that there is no \emph{potential
}solution of \eqref{eq:xi_overdet_EP-1}, unless $\OS$ is a ball.
\\
We also remark that from $\nabla Q$ a potential solution of \eqref{eq:xi_overdet_EP-1},
it is possible to go ,,up'' to the Schiffer problem, i.e. then $-\mu Q+C=\lambda\Delta Q$,
$Q|_{\dOS}=c$ and $\partial_{n}Q|_{\dOS}=0$ for some constants $C,c\in\mathbb{R}$,
so $Q-\frac{C}{\mu}$ solves the Schiffer problem.
\end{rem}

\begin{prop}
Theorem \ref{thm:Schiffer} also holds if \eqref{eq:xi_overdet_EP-1}
is replaced by the corresponding system of wave equations 
\begin{align}
\begin{split}\begin{cases}
\begin{array}{rcll}
-\Delta\psi & = & \mu\psi & \text{in }\OS,\\
\psi & = & 0 & \text{on }\dOS,\\
(\nabla\psi)n & = & qn & \text{on }\dOS.
\end{array}\end{cases}\end{split}
\label{eq:xi_WAVE}
\end{align}
\end{prop}

\begin{proof}
The proof follows the same constructions as for Theorem \ref{thm:Schiffer},
by taking the divergence or the gradient, respectively. The calculations
are slightly simpler. 
\end{proof}
\begin{cor}
The ball is the only bad domain if and only if the Schiffer Conjecture
holds. Bad domains are real analytic. The good domains in Proposition
\ref{prop:ATgooddomains} are not Schiffer domains. 
\end{cor}

The calculations in the proof of Theorem \ref{thm:Schiffer} show
the following sufficient conditions for $\OS$ to be a ball that were
also obtained in \cite{Liu2007JMAASchiffer,LiuJDE2007Schiffer} for
the three-dimensional and in \cite{KawohlLucia2020} for the two-dimensional
case.
\begin{cor}
\label{cor:3neumannBC}If $u$ satisfies \eqref{eq:schifferProblem}
and the third-order Neumann condition $\partial_{n}^{3}u|_{\dOS}=C_{S}$
for some constant $C_{S}\geq0$, then $\OS$ is a ball. 
\end{cor}

\begin{proof}
Let $u$ be a solution of \eqref{eq:schifferProblem}. Then by Theorem
\ref{thm:Schiffer}, $\xi=\nabla u$ solves \eqref{eq:xi_overdet_EP-1}.
Then by \eqref{eq:SH}, for all $x\in\dOS$, 
\[
B(x)=\Delta\partial_{n}u(x)-\partial_{n}^{3}u(x)=-\frac{q}{\lambda}H(x),
\]
where $H(x)$ is the mean curvature of $\dOS$ at $x$. Moreover,
by approximating from the inside with $(x^{j})\subset\OS$, using
the PDE then the Neumann boundary condition in the limit, 
\[
\Delta\partial_{n}u(x)=\lim_{j\to\infty}-\mu_{S}\partial_{n}u(x^{j})=0.
\]
This shows that if $\partial_{n}^{3}u(x)=C_{S}$ , then $H(x)=\frac{\lambda}{q}C_{S}=-\frac{C_{S}}{\mu_{S}c_{S}}$
is constant. By Alexandrov's Theorem,\cite{Aleksandrov1962Spheres},
$\OS$ is a ball. 
\end{proof}

\section{\label{sec:nonlinear}Consequences for related non-linear systems}

Theorem \ref{thm:strongstab} gives a complete picture of the long-time
behaviour of the fluid-elastic linear semigroup. The next step is
to apply it in a nonlinear setting. We consider two corresponding
nonlinear systems. The first system still keeps the fluid domain fixed,
but introduces fluid convection. With respect to elastic deformation,
this can be seen as a high-oscillation - small deformation set-up.
In the next subsection, it is shown that, maybe not surprisingly the
long-term behaviour can be identified completely from initial data
also in this case. We note that we move to a stronger notion of solutions
then obtained from the semigroup, in order to be able to handle the
nonlinearity in the proof of existence of solutions \cite{DisserLuckas2025}.
The second system considered in Subsection \ref{subsec:fullynonlinear}
fully incorporates the interaction in the change of fluid domain according
to the displacement of the structure. In this case, pressure waves
are also relevant, but the situation is less clear.

\subsection{\label{subsec:resultsluckas}Identification of the limit pressure
wave for the high oscillation - small deformation system}

Consider the system
\begin{eqnarray}
\begin{cases}
\begin{array}{rcll}
\dot{u}+(u\cdot\nabla)u-\Div(S(u,p)) & = & 0 & \textrm{in }(0,T)\times\OF,\\
\Div(u) & = & 0 & \textrm{in }(0,T)\times\OF,\\
S(u,p)n & = & L(\xi)n & \textrm{on }(0,T)\times\dOS,\\
u & = & \dot{\xi} & \textrm{on }(0,T)\times\dOS,\\
u & = & 0 & \text{on }(0,T)\times\dO,\\
\ddot{\xi}-\textrm{div}(L(\xi)) & = & 0 & \textrm{in }(0,T)\times\OS,\\
u(0) & = & u_{0} & \textrm{in }\OF,\\
\xi(0) & = & \xi_{0} & \textrm{in }\OS,\\
\dot{\xi}(0) & = & \xi_{1} & \textrm{in }\OS,
\end{array}\end{cases}\label{eq:nonlin_system}
\end{eqnarray}
which is different from \eqref{eq:linsys} in that it includes fluid
convection and does not include a shift in the Lamé operator. In \cite{DisserLuckas2025},
we proved the existence of global strong solutions for small initial
data and convergence to rest in the fluid velocity, and we characterized
the long-time behaviour of $\xi$. In particular, in the case of a
bad domain $\OS$, we proved convergence to pressure wave solution
$\eta^{*}$ which was constructed, in dependence of the full orbits
$(\xi(t),\dot{\xi}(t))_{t\geq0}$. By adapting the invariance from
Lemma \ref{lem:invariance} to this system, we can now characterize
$\eta^{*}$ explicitly from the initial data, in Corollary \ref{cor:nonlinearPressureWave}.
To fix the details, we recall the existence of strong solutions from
\cite{DisserLuckas2025}.
\begin{thm}
\label{thm:globEx} 
\begin{enumerate}
\item \emph{Local Existence:} For any initial data 
\begin{equation}
\xi_{0}\in\H1^{2}(\OS),\,\xi_{1}\in\H1^{1}(\OS),\,u_{0}\in\H1^{2}(\OF),\label{idata}
\end{equation}
such that there exist 
\begin{align}
u_{1},\,p_{0}\in\H1^{1}(\OF)\text{ and }\xi_{2}\in\Lp^{2}(\OS),\label{idata_2}
\end{align}
satisfying the compatibility conditions
\begin{equation}
\begin{cases}
\begin{array}{rcll}
u_{1}+(u_{0}\cdot\nabla)u_{0}-\Div(S(u_{0},p_{0})) & = & 0 & \text{in }\OF,\\
\Div(u_{0}) & = & 0 & \text{in }\OF,\\
\Div(u_{1}) & = & 0 & \text{in }\OF,\\
S(u_{0},p_{0})n & = & L(\xi_{0})n & \textrm{on }\ensuremath{\partial\Omega_{S}},\\
u_{0} & = & \xi_{1} & \text{on }\dOS,\\
u_{0} & = & 0 & \text{on }\dO,\\
u_{1} & = & 0 & \text{on }\dO,\\
\xi_{2}-\Div(L(\xi_{0})) & = & 0 & \text{in }\OS,
\end{array}\end{cases}\label{eq:compatibility_thm}
\end{equation}
there is a time 
\[
T=T\left(\Vert u_{0}\Vert_{\H1^{1}(\OF)},\Vert u_{1}\Vert_{\Lp^{2}(\OF)},\Vert\varepsilon(\xi_{0})\Vert_{\Lp^{2}(\OS)},\Vert\xi_{1}\Vert_{\H1^{1}(\OS)},\Vert\xi_{2}\Vert_{\Lp^{2}(\OS)}\right)>0
\]
such that a unique strong solution 
\begin{align*}
u\in & \C^{0}([0,T];\H1^{2}(\OF))\cap\H1^{1}(0,T;\H1^{1}(\OF))\cap\C^{1}([0,T];\Lp^{2}(\OF)),\\
p\in & \C^{0}([0,T];\H1^{1}(\OF)),\\
\xi\in & \C^{0}([0,T];\H1^{2}(\OS))\cap\C^{1}([0,T];\H1^{1}(\OS))\cap\C^{2}([0,T];\Lp^{2}(\OS)),
\end{align*}
of (\ref{eq:nonlin_system}) exists. 
\item \emph{\label{enu:GlobEx}Global Existence: }The total of kinetic and
elastic energies for system \eqref{eq:nonlin_system} is defined by
\[
E(t):=\frac{1}{2}\Vert u(t)\Vert_{\Lp^{2}(\OF)}^{2}+\frac{1}{2}\Vert\dot{\xi}(t)\Vert_{\Lp^{2}(\OS)}^{2}+\int_{\OS}L(\xi):\nabla(\xi)(t)\,\dx y,
\]
and a corresponding higher-order quantity is 
\[
F(t):=\frac{1}{2}\Vert\dot{u}(t)\Vert_{\Lp^{2}(\OF)}^{2}+\frac{1}{2}\Vert\ddot{\xi}(t)\Vert_{\Lp^{2}(\OS)}^{2}+\int_{\OS}L(\dot{\xi}):\nabla(\dot{\xi})(t)\,\dx y.
\]
There exist constants $C_{u}>0,\,C_{E}>0,\,C_{K}>0$ such that for
any initial data as in \eqref{idata}, \eqref{idata_2} satisfying
the compatibility conditions (\ref{eq:compatibility_thm}) and the
bounds 
\begin{align}
\Vert u_{0}\Vert_{\H1^{1}(\OF)}\leq C_{u},\hspace{0.5cm}E(0)\leq C_{E},\hspace{0.5cm}F(0)\leq C_{F},\label{inbounds}
\end{align}
the corresponding unique solution $(u,p,\xi)$ to (\ref{eq:nonlin_system})
exists up to any time $T>0$. 
\end{enumerate}
\end{thm}

The long-time-behaviour of solutions that we found in \cite{DisserLuckas2025}
is the following:
\begin{thm}
\label{thm:longtimenonlinear} Define $q(t):=\frac{1}{\vert\OF\vert}\int_{\OF}p(t,y).$Then
the global solution $(u,p,\xi)$ to (\ref{eq:nonlin_system}) given
in Theorem \ref{thm:globEx}\eqref{enu:GlobEx} satisfies 
\begin{align*}
\lim_{t\to\infty}\Vert u(t)\Vert_{\H1^{1}(\OF)}\qquad & =0,\\
\lim_{t\to\infty}\Vert(p-q)(\cdot+t)\Vert_{\Lp^{2}((0,T/2)\times\OF)} & =0,\\
\lim_{t\to\infty}\Vert L(\xi(\cdot+t)n-q(\cdot+t)n\Vert_{\Lp^{2}(\H1^{-1/2}(\dOS))} & =0,
\end{align*}
and 

\begin{align}
\lim_{t\to\infty}\Vert\xi(t)-\eta^{*}(t)-r(t)-\kappa_{0}\mathrm{Id}_{\mathbb{R}^{3}}\Vert_{\H1^{1}(\OS)}=0,\label{eq:xi_conv_baddomain-1}
\end{align}
where $\kappa_{0}=\frac{K_{\xi_{0}}}{3|\OS|}$, $\xi(t)-\eta^{*}(t)-r(t)\in\bar{H}^{1}(\OS)$
for all $t\geq0$, 
\begin{align*}
r(t)\in R=\text{ker}(D)=\{r:\mathbb{R}^{3}\to\mathbb{R}^{3}:r(y)=My+b,\,M=-M^{T}\in\mathbb{R}^{3\times3},\,b\in\mathbb{R}^{3}\},
\end{align*}
is the projection of $\xi(t)-\eta^{*}(t)$ onto the set $R$ of rigid
motions, and either $\eta^{*}=0$ or $\eta^{*}$ is a pressure wave.
\end{thm}

\begin{proof}
With respect to Theorem 13 in \cite{DisserLuckas2025}, the only change
is in the characterization of the time-constant displacement $\kappa_{0}\mathrm{Id}_{\mathbb{R}^{3}}$.
Compared to the linear system \eqref{eq:linsys}, in the present system
\eqref{eq:nonlin_system}, there is no shift in the Lamé equations.
Thus, the rigid motions remain in the kernel of the Lamé operator
and must be accounted for (both cases with and without shift are addressed
in the present paper because they may be of independent interest).
This also has the consequence that for all $\xi\in H^{1}(\OS)$, the
projection onto $\bar{H}^{1}$ along the kernel of this operator is
determined via solutions of the Neumann problem 

\begin{align}
\begin{cases}
\begin{array}{rcll}
\Div(L(\varphi)) & = & 0 & \text{ in }\OS,\\
L(\varphi)n & = & \kappa n & \text{ on }\dOS,
\end{array}\end{cases}\label{eq:xi_stat_system-1-1}
\end{align}
with parameter $\kappa\in\mathbb{R}$. Up to an additive constant,
this solution is given explicitly by $\varphi(y)=\varphi_{0}y,\varphi_{0}=\frac{\kappa}{\lambda_{0}+3\lambda_{1}}$.
In particular, using that $\OS$ is bounded, 
\begin{equation}
\xi\in H^{1}(\OS)\Leftrightarrow\xi-\varphi_{0}\mathrm{Id}_{\OS}\in\bar{H}^{1}(\OS),\varphi_{0}=\frac{K_{\xi}}{3|\OS|}.\label{eq:Kxigone-1}
\end{equation}
Note that $\eta^{*}(t)\in\bar{H}^{1}(\OS)$ and we calculate that
$r(t,y)=My+b\in\Es$ due to 
\[
b\cdot\int_{\OS}\psi_{k}=-\frac{b}{\mu_{k}}\cdot\int_{\OS}\mathrm{div}(\sigma(\psi_{k}))=-\frac{bq_{k}}{\mu_{k}}\cdot\int_{\dOS}n=0
\]
and 
\begin{align*}
\int_{\OS}My\cdot\psi_{k}(y) & =-\frac{1}{\mu_{k}}\cdot\int_{\OS}My\cdot\mathrm{div}(\sigma(\psi_{k}))\\
 & =\frac{1}{\mu_{k}}\cdot\int_{\OS}M:\sigma(\psi_{k})-\frac{q_{k}}{\mu_{k}}\cdot\int_{\dOS}My\cdot n=-\frac{q_{k}}{\mu_{k}}\int_{\OS}\mathrm{div}(My)=0
\end{align*}
for all $k\in K$, due to the skew-symmetry of $M$.
\end{proof}
In \cite{DisserLuckas2025}, $\eta^{*}$ was given only implicitly,
through a compactness argument. Now it can be determined from the
initial data. 
\begin{cor}
\label{cor:nonlinearPressureWave} In Theorem \ref{thm:longtimenonlinear},
$\eta^{*}=\eta^{\Ws}$, where $\eta^{\Ws}(t)=\eta\left(t,\mathcal{P}_{\Ws}\mathbf{x}_{0}\right)$
is the pressure wave eminating from the projection of the initial
data onto $\Ws$. 
\end{cor}

\begin{proof}
The key observation is that also for the semiflow associated to system
\eqref{eq:nonlin_system}, Lemma \ref{lem:invariance} holds. We express
this in the following form: For all $t\geq0$, denote by $\bar{\xi}(t)=\Xi(t,\mathbf{x}_{0})$
the displacement component of the strong solution of \eqref{eq:nonlin_system}
corresponding to initial data $\mathbf{x}_{0}=(\xi_{0}-\kappa_{0}\mathrm{Id}_{\mathbb{R}^{3}},\xi_{1},u_{0})$.
Then 
\begin{equation}
\mathcal{P}_{E}\Xi(t,\mathbf{x}_{0})=\Xi(t,\mathcal{P}_{\Es}\mathbf{x}_{0}),\qquad\mathcal{P}_{W}\Xi(t,\mathbf{x}_{0})=\Xi(t,\mathcal{P}_{\Ws}\mathbf{x}_{0}),\label{eq:invNL}
\end{equation}
follow exactly as in the proof of Lemma \ref{lem:invariance}, where
regarding the fluid component $u(t)$, only $\mathrm{div}u(t)=0$
in $\OF$, $\dot{\xi}(t)=u(t)$ on $\dOS$ and $u(t)=0$ on $\partial\Omega$
was used (regarding the next subsection, it is important to note that
it was also used that the fluid domain does not change). Now due to
\eqref{eq:xi_conv_baddomain-1} and \eqref{eq:invNL},
\begin{align*}
\mathcal{P}_{E}\left(\Xi(t)\right) & \overset{t\to\infty}{\to}\mathcal{P}_{E}\left(\eta^{*}(t)+r(t)\right)=r(t),\text{ and }\\
\eta^{\Ws}(t)=\Xi(t,\mathcal{P}_{\Ws}\mathbf{x}_{0})=\mathcal{P}_{W}\left(\Xi(t)\right) & \overset{t\to\infty}{\to}\mathcal{P}_{W}\left(\eta^{*}(t)+r(t)\right)=\eta^{*}(t)\in\Ws.
\end{align*}
By choosing e.g.~the sequences $t_{n}^{k}:=\sqrt{1+\mu_{k}}\frac{\pi}{2}n$
and by comparing the coefficients $\eta_{k}^{\Ws}(t_{n}^{k})$ and
$\eta_{k}^{*}(t_{n}^{k})$, we see that due to characterization \eqref{eq:charomega},
the convergence $\eta^{\Ws}(t)\overset{t\to\infty}{\to}\eta^{*}(t)$
implies the equality of both functions.
\end{proof}
\begin{rem}
The conditions for global existence of solutions in Theorem \ref{thm:longtimenonlinear}
can now be slightly relaxed: for any compatible initial data $(\xi_{0},\xi_{1},u_{0})$,
it is sufficient that, after subtracting $\kappa_{0}\mathrm{Id}_{\mathbb{R}_{3}}$
from $\xi_{0}$, their projection onto $\Es$ satisfies the smallness
conditions \eqref{inbounds}. Regardless of its size, the projected
initial data $\mathcal{P}_{\Ws}(\xi_{0}-\kappa_{0}\mathrm{Id}_{\mathbb{R}^{3}},\xi_{1},u_{0})$
will only create a pressure wave solution with constant energy and
no fluid component. It exists globally and can be subtracted or added. 
\end{rem}

\subsection{\label{subsec:fullynonlinear}Consequences for the fully non-linear
system}

Finally, we briefly consider the fully coupled system

\begin{eqnarray}
\begin{cases}
\begin{array}{rcll}
\dot{u}+(u\cdot\nabla)u-\Div(S(u,p)) & = & 0 & \textrm{in }(0,T)\times\OF(0,T),\\
\Div(u) & = & 0 & \textrm{in }(0,T)\times\OF(0,T),\\
(S(u,p)\circ X)\mathrm{Cof}(\nabla X)n & = & L(\xi)n & \textrm{on }(0,T)\times\dOS,\\
u\circ X & = & \dot{\xi} & \textrm{on }(0,T)\times\dOS,\\
u & = & 0 & \text{on }(0,T)\times\dO,\\
\ddot{\xi}-\textrm{div}(L(\xi)) & = & 0 & \textrm{in }(0,T)\times\OS,\\
\dot{X} & = & u\circ X & \text{in }(0,T)\times\OF,
\end{array}\end{cases}\label{eq:fullsystem-1}
\end{eqnarray}
with initial conditions 
\begin{eqnarray*}
u(0)=u_{0}\quad\text{ and} & X(0)=\mathrm{Id}, & \text{in }\OF,\\
\xi(0)=\xi_{0}\quad\text{ and} & \dot{\xi}(0)=\xi_{1}, & \text{in }\OS.
\end{eqnarray*}
Here, the fluid domain is time-dependent $\OF(t)=\Omega\setminus\overline{\OS(t)}$,
adapting to the elastic displacement with $X(t,y)=y+\xi(t,y)$ on
$\dOS$. The existence of local-in-time strong solutions to this system
was shown in \cite{BGT2019,RayVan14FSI}. Without additional damping
on the structure, the existence of global solutions, even for small
data, is an open problem. With additional damping terms, there are
global existence results \cite{IKLT2017,KO2023}, also in the case
of nonlinear elasticity \cite{BKS2024}.

It is straightforward to check that pressure waves $\eta$ with $u(t,y)=0,p(t,y)=p(t)$
provide solutions to system \eqref{eq:fullsystem-1}. This is due
to the fact that if $u\equiv0$, the fluid domain $\OF$ remains fixed
and thus system \eqref{eq:fullsystem-1} reduces to \eqref{eq:linsys}.
Consequently, the space $\Ws$ is invariant also under this nonlinear
semiflow. This must be accounted for in a global analysis. However,
the calculations in \eqref{eq:odeXik}, \eqref{eq:odezetak} no longer
work and an analogue of the decomposition in Lemma \ref{lem:invariance}
needs further investigations. 

\bibliographystyle{alpha}
\bibliography{literaturFSI}

\end{document}